\UseRawInputEncoding
\documentclass[12pt]{article}

\usepackage{dsfont}
\usepackage{amsfonts,amsmath,amsthm}
\usepackage{amssymb}
\usepackage{algorithm}
\usepackage{algpseudocode}
\usepackage{xcolor}
\usepackage{graphicx}
\usepackage{stmaryrd}        
\usepackage{multirow}

\usepackage[paper=a4paper,dvips,top=2cm,left=2cm,right=2cm,
foot=1cm,bottom=4cm]{geometry}

\newcommand{\red}[1]{\begin{color}{red}#1\end{color}}

\newcommand{\ve}{{\bf e}}
\newcommand{\vh}{{\bf h}}

\newcommand{\supp}{\text{supp}}

\begin{document}
	\large
	
	\title{Nonnegative Biquadratic Tensors}
\author{ 	Chunfeng Cui\footnote{LMIB of the Ministry of Education, School of Mathematical Sciences, Beihang University, Beijing 100191 China.
		({\tt chunfengcui@buaa.edu.cn}).}
	\and { \  Liqun Qi\footnote{Department of Applied Mathematics, The Hong Kong Polytechnic University, Hung Hom, Kowloon, Hong Kong.
			({\tt maqilq@polyu.edu.hk}).}
	}
}
\date{\today}
\maketitle

\begin{abstract}
An M-eigenvalue of a nonnegative biquadratic tensor is   referred to as an M$^+$-eigenvalue if it has a pair of nonnegative M-eigenvectors.  If furthermore that pair of M-eigenvectors is positive, then that M$^+$-eigenvalue is called an M$^{++}$-eigenvalue.  A nonnegative biquadratic tensor has at least one M$^+$ eigenvalue,
{and} the largest M$^+$-eigenvalue is {both} the  largest M-eigenvalue and the M-spectral radius.
{For} irreducible nonnegative biquadratic {tensors}, all the  M$^+$-eigenvalues   are M$^{++}$-eigenvalues.   Although the M$^+$-eigenvalues of irreducible nonnegative biquadratic tensors are not unique in general, we establish a sufficient condition to ensure their uniqueness.  For an irreducible nonnegative biquadratic tensor, the largest M$^+$-eigenvalue has a max-min characterization, while the smallest M$^+$-eigenvalue has a min-max characterization.
A Collatz algorithm for computing the largest   M$^+$-eigenvalues is proposed.  Numerical results are reported.

\medskip



\end{abstract}

\renewcommand{\Re}{\mathds{R}}
\newcommand{\rank}{\mathrm{rank}}
\newcommand{\X}{\mathcal{X}}
\newcommand{\A}{\mathcal{A}}
\newcommand{\I}{\mathcal{I}}
\newcommand{\B}{\mathcal{B}}
\newcommand{\PP}{\mathcal{P}}
\newcommand{\C}{\mathcal{C}}
\newcommand{\D}{\mathcal{D}}
\newcommand{\LL}{\mathcal{L}}
\newcommand{\OO}{\mathcal{O}}
\newcommand{\e}{\mathbf{e}}
\newcommand{\0}{\mathbf{0}}
\newcommand{\1}{\mathbf{1}}
\newcommand{\dd}{\mathbf{d}}
\newcommand{\ii}{\mathbf{i}}
\newcommand{\jj}{\mathbf{j}}
\newcommand{\kk}{\mathbf{k}}
\newcommand{\va}{\mathbf{a}}
\newcommand{\vb}{\mathbf{b}}
\newcommand{\vc}{\mathbf{c}}
\newcommand{\vq}{\mathbf{q}}
\newcommand{\vg}{\mathbf{g}}
\newcommand{\pr}{\vec{r}}
\newcommand{\pc}{\vec{c}}
\newcommand{\ps}{\vec{s}}
\newcommand{\pt}{\vec{t}}
\newcommand{\pu}{\vec{u}}
\newcommand{\pv}{\vec{v}}
\newcommand{\pn}{\vec{n}}
\newcommand{\pp}{\vec{p}}
\newcommand{\pq}{\vec{q}}
\newcommand{\pl}{\vec{l}}
\newcommand{\vt}{\rm{vec}}
\newcommand{\x}{\mathbf{x}}
\newcommand{\vx}{\mathbf{x}}
\newcommand{\vy}{\mathbf{y}}
\newcommand{\vu}{\mathbf{u}}
\newcommand{\vv}{\mathbf{v}}
\newcommand{\y}{\mathbf{y}}
\newcommand{\vz}{\mathbf{z}}
\newcommand{\T}{\top}
\newcommand{\R}{\mathcal{R}}

\newtheorem{Thm}{Theorem}[section]
\newtheorem{Def}[Thm]{Definition}
\newtheorem{Ass}[Thm]{Assumption}
\newtheorem{Lem}[Thm]{Lemma}
\newtheorem{Prop}[Thm]{Proposition}
\newtheorem{Cor}[Thm]{Corollary}
\newtheorem{example}[Thm]{Example}
\newtheorem{remark}[Thm]{Remark}

\section{Introduction}

A real fourth order  $(m \times n \times m \times n)$-dimensional tensor $\A = (a_{i_1j_1i_2j_2}) \in \Re^{m \times n \times m \times n}$ is called a {\bf biquadratic tensor}, where $m$ and $n$ are integers with $m, n \ge 2$.   Denote $[n] := \{ 1, {\dots,} n \}$.  If for {any} $i_1, i_2 \in [m]$ and $j_1, j_2\in [n]$,
$$a_{i_1j_1i_2j_2} = a_{i_2j_2i_1j_1},$$
then  $\A$ is called a weakly symmetric tensor.  If furthermore for {any} $i_1, i_2 \in [m]$ and $j_1, j_2\in [n]$,
$$a_{i_1j_1i_2j_2} = a_{i_2j_1i_1j_2}{=a_{i_1j_2i_2j_1}},$$
then $\A$ is called a symmetric biquadratic tensor.
Let $BQ(m, n)$ denote   the set of all biquadratic tensors in $\Re^{m \times n \times m \times n}$.
{This set forms} a linear space.
Additionally, we also denote by $NBQ(m,n)$  the set of nonnegative $(m\times n\times m\times n)$-dimensional biquadratic tensors.

We say that a {biquadratic} tensor {$\A\in BQ(m,n)$} is {\bf positive semi-definite} if for any $\x \in \Re^m$ and $\y \in \Re^n$,
\begin{equation}\label{equ:PSD}
{f(\x, \y) \equiv} \langle \A, \x \circ \y \circ \x \circ \y \rangle \equiv \sum_{i_1, i_2 =1}^m \sum_{j_1, j_2 = 1}^n a_{i_1j_1i_2j_2}x_{i_1}y_{j_1}x_{i_2}y_{j_2} \ge 0,
\end{equation}
and it is {\bf positive definite} if for any $\x \in \Re^m, \x^\top \x = 1$ and $\y \in \Re^n, \y^\top \y = 1$,
$${f(\x, \y) > 0.}$$

In 2009, Qi, Dai and Han \cite{QDH09} introduced M-eigenvalues and M-eigenvectors for symmetric biquadratic tensors {during their investigation of strong ellipticity condition of the {elastic} tensor in solid mechanics.}   It was proved there that a symmetric biquadratic tensor always has M-eigenvalues, and it is positive semi-definite (definite) if and only if all of its M-eigenvalues are nonnegative (positive).    Since then, more than one hundred papers on symmetric biquadratic tensors and M-eigenvalues appeared \cite{CCZ21, DLQY20, HLW20, LLL19, LCLL22, QHZX21, WSL20, WQZ09, YY12, Zh23, ZLS24}.

Very recently, Qi and Cui \cite{QC25} generalized M-eigenvalues to  general (nonsymmetric) biquadratic tensors.  This {definition} was motivated by the study of covariance tensors in statistics \cite{CHHS25}, which are not symmetric {(only weakly symmetric)}, {yet remains} positive semi-definite.  Again, a general (nonsymmetric) biquadratic tensor always has M-eigenvalues, and it is positive semi-definite (definite) if and only if all of its M-eigenvalues are nonnegative (positive).  Notably, {the definition of M-eigenvalues and M-eigenvectors for symmetric biquadratic tensors} is a special case of this new definition.

{Nonnegative tensors, along with  positive semidefinite tensors, completely
	positive tensors, and copositive tensors, constitute  four important classes of special cubic tensors \cite{QL17}.}
In this paper, we focus on nonnegative biquadratic tensors,   which  were first studied in \cite{DLQY20}.    Nonnegative biquadratic tensors arising from bipartite graphs,  polynomial theory and the study of M-biquadratic tensors, as described in \cite{QC25}.  In \cite{QC25}, M-biquadratic tensors were introduced. {It was} proved that a general biquadratic tensor always has M-eigenvalues, and it is positive semi-definite (definite) if and only if all of its M-eigenvalues are nonnegative (positive) positive semi-definite.  A key step in identifying an   M-biquadratic tensor is to calculate the largest M-eigenvalue of a nonnegative biquadratic tensor.
These theoretical developments and practical motivations have driven our interest in exploring  nonnegative biquadratic tensors and their M-eigenvalues.

In the next section, we review the definition of M-eigenvalues of general biquadratic tensors, and introduce {M$^+$-eigenvalues} and M$^{++}$-eigenvalues for nonnegative biquadratic tensors.  {Specifically,} an M-eigenvalue of a nonnegative biquadratic tensor is called an M$^+$-eigenvalue if it has a pair of nonnegative M-eigenvectors. {If the M-eigenvectors are strictly positive, the eigenvalue is further classified as}  an M$^{++}$-eigenvector.

In Section 3, we show that a nonnegative biquadratic tensor always has at least one M$^+$-eigenvalue.  While M$^+$-eigenvalues are not unique in general, we present a sufficient condition ensuring the uniqueness of M$^+$-eigenvalue and show that the largest M$^+$-eigenvalue of a nonnegative biquadratic tensor is also \red{both} the largest M-eigenvalue and M-spectral radius of that tensor.  Furthermore, we prove that all of the M$^+$-eigenvalues of an irreducible nonnegative biquadratic tensor are M$^{++}$-eigenvalues of that tensor.

We then show in Section 4 that the largest M$^+$-eigenvalue of an irreducible nonnegative biquadratic tensor has a max-min characterization, while the smallest M$^+$-eigenvalue of that irreducible nonnegative biquadratic tensor has a min-max characterization.
{These results show that the Collatz min-max characterization of irreducible nonnegative  matrices is only partially true.}


A   {Collatz-type} algorithm for computing the largest M$^+$-eigenvalue of an irreducible nonnegative biquadratic tensor is proposed in Section 5.
Numerical results are reported in Section 6 {to demonstrate the effectiveness of the proposed algorithm.
Finally, in Section 7, we provide some concluding remarks and discuss potential directions for future research.}    


%


\section{M$^+$-Eigenvalues and M$^{++}$-Eigenvalues of Nonnegative Biquadratic Tensors}

Suppose that $\A = (a_{i_1j_1i_2j_2}) \in BQ(m, n)$.   A   {real} number $\lambda$ is called an  {M-eigenvalue} of $\A$ if there are  {real} vectors  $\x = (x_1, {\dots,} x_m)^\top \in {\Re}^m, \y = (y_1, {\dots,} y_n)^\top \in {\Re}^n$ such that the following equations are satisfied:
For  \red{any} $i {\in [m]}$,
\begin{equation} \label{e5}
\sum_{i_1=1}^m \sum_{j_1, j_2=1}^n a_{i_1j_1ij_2}x_{i_1}y_{j_1}y_{j_2} +	\sum_{i_2=1}^m \sum_{j_1, j_2=1}^n a_{ij_1i_2j_2}y_{j_1}x_{i_2}y_{j_2} = 2\lambda x_i;
\end{equation}
 for  \red{any} $j {\in [n]}$,
\begin{equation} \label{e6}
\sum_{i_1,i_2=1}^m\sum_{j_1=1}^n a_{i_1j_1i_2j}x_{i_1}y_{j_1}x_{i_2} +	\sum_{i_1,i_2=1}^m\sum_{j_2=1}^n a_{i_1ji_2j_2}x_{i_1}x_{i_2}y_{j_2} = 2\lambda y_j;
\end{equation}
and
\begin{equation} \label{e7}
\x^\top \x ={ \y^\top \y} = 1.
\end{equation}
Then $\x$ and $\y$ are called the corresponding  {M-eigenvectors}.
We may   rewrite equations \eqref{e5} and \eqref{e6} as
	$$\frac12\A\cdot\vy\vx\vy+\frac12\A\vx\vy\cdot\vy=\lambda \vx \ \mathrm{ and }\ \frac12\A\vx\cdot\vx\vy+\frac12\A\vx\vy\vx\cdot=\lambda\vy,$$
	  respectively.
Furthermore, equations \eqref{e5}-\eqref{e7} are corresponding to the KKT system of the following optimization problem
\begin{equation}
	\min_{\vx\in\Re^m,\vy\in\Re^n}\ \A\vx\vy\vx\vy \ \  \text{s.t.}\ \  \x^\top \x =\y^\top \y = 1,
\end{equation}
{and $\lambda$ is the corresponding Lagrange multiplier.}


{When $\A$ is symmetric, equations \eqref{e5} and \eqref{e6} reduce to the following equations} {proposed} in 2009 by Qi, Dai and Han \cite{QDH09} {for symmetric biquadratic tensors},  i.e,
for {any} $i {\in [m]}$,
\begin{equation} \label{e5_sym}
\sum_{i_2=1}^m \sum_{j_1, j_2=1}^n a_{ij_1i_2j_2}y_{j_1}x_{i_2}y_{j_2} =  \lambda x_i;
\end{equation}
{for} \red{any} $j {\in [n]}$,
\begin{equation} \label{e6_sym}
\sum_{i_1,i_2=1}^m\sum_{j_2=1}^n a_{i_1ji_2j_2}x_{i_1}x_{i_2}y_{j_2} =  \lambda y_j.
\end{equation}
{We may also rewrite equations \eqref{e5_sym} and \eqref{e6_sym} as $\A\cdot\vy\vx\vy=\lambda \vx$ and $\A\vx\cdot\vx\vy=\lambda\vy$, respectively.}
Subsequently, several numerical methods, including the WQZ method \cite{WQZ09}, semi-definite relaxation method \cite{LNQY10,YY12}, and the shifted inverse  power method \cite{ZLS24}, were proposed {for computing M-eigenvalues of symmetric biquadratic tensors.}

The following theorem was established in \cite{QC25}.

\begin{Thm} \label{T2.1}
Suppose that $\A \in BQ(m, n)$.  Then $\A$ always {has} M-eigenvalues.  Furthermore, $\A$ is positive semi-definite if and only if all of its M-eigenvalues are nonnegative, {and} $\A$ is positive definite if and only if all of its M-eigenvalues are positive.
\end{Thm}

By definition and {the preceding} theorem, the set of M-eigenvalues of a biquadratic tensor is {both} nonempty and compact.   {Consequently}, for {any} biquadratic tensor, there are the largest and the smallest M-eigenvalues.

Let $\A \in NBQ(m, n)$.   Suppose that $\lambda$ is an M-eigenvalue of $\A$ {associated} with a pair of nonnegative M-eigenvectors $\x \in \Re_+^m$ and $\y \in \Re_+^n$.   Then $\lambda$ is also nonnegative, i.e., $\lambda \ge 0$.    We call $\lambda$ {an} M$^+$-eigenvalue of $\A$. Furthermore, if both $\x$ and $\y$ are positive, we call $\lambda$ an M$^{++}$-eigenvalue of $\A$.
{As we will see later,} M$^+$-eigenvalues and M$^{++}$-eigenvalues play {an important} role for nonnegative biquadratic tensors.

\section{Weak Perron-Frobenius Theorem}

{\subsection{Properties of irreducible nonnegative biquadratic tensors}}
Among the study of  nonnegative  tensors, irreducibility plays   a significant role.
Following the definitions in \cite{CQZ10, DLQY20}, we say a biquadratic tensor is irreducible if  for all $i\in [m]$ and $j\in [n]$, the   matrices $A_{\vx}^{(j,j)}={\frac12\A(:,j,:,j)+\frac12\A(:,j,:,j)^\top}\in \Re^{m\times m}$ and   $A_{\vy}^{(i,i)}={\frac12\A(i,:,i,:)+\frac12\A(i,:,i,:)^\top}\in \Re^{n\times n}$ are all irreducible.
For any $\vx\in\Re^m$ and  $\vy\in\Re^n$, denote $\supp(\vx)=\{i\in[m]: x_i\neq 0\}$ and $\supp(\vy)=\{j\in[n]: y_j\neq 0\}$, respectively.
{We also denote}  $\ve_i(m)$ and $\ve_j(n)$ as the $i$th and $j$th unit vectors in $\Re^m$ and  $\Re^n$, respectively.
An irreducible nonnegative biquadratic tensor has the following property.

\begin{Thm}\label{Thm_nonegative_supp}
Suppose that $\A = \left(a_{i_1j_1i_2j_2}\right) \in NBQ(m, n)$, where $m,n\ge 2$.  For any {nonzero vectors} $\vx\in\Re_+^m$ and $\vy\in\Re_+^n$, let $$\vu ={1 \over 2}(\A+\I)\vx\vy\cdot \vy + {1 \over 2}(\A+\I)\cdot \vy\vx\vy{\in\Re^m}$$  and $$\vv = {1 \over 2}(\A + \I) \vx\vy\vx\cdot +{1 \over 2}(\A+\I) \vx\cdot\vx\vy{\in\Re^n}.$$ Then we have
\begin{equation*}
	\supp(\vx) {\subseteq} \supp(\vu) \text{ and } \supp(\vy) {\subseteq}  \supp(\vv).
\end{equation*}
Furthermore, $\A$ is {reducible} if and only if  there is  $\vx\in\Re_+^m$ with $0< |\supp(\vx)|<m$  and $\vy\in\Re_+^n$ with $0< |\supp(\vy)| {<} n$ such that {either} $\supp(\vx) =\supp(\vu)$  or
$\supp(\vy) = \supp(\vv)$.

\end{Thm}
\begin{proof}
Let $\vx\in\Re_+^m$ and $i\in\supp(\vx)$. Then $x_i>0$ and
\begin{equation*}
	u_i={1 \over 2}\sum_{i_1=1}^m \sum_{j_1, j_2=1}^n a_{i_1j_1ij_2}x_{i_1}y_{j_1}y_{j_2} + {1 \over 2}\sum_{i_2 = 1}^m\sum_{j_1, j_2 = 1}^n a_{ij_1i_2j_2} y_{j_1}x_{i_2}y_{j_2} + \sum_{j=1}^n x_iy_j^2 >0.
\end{equation*}
Thus, we have $	\supp(\vx) {\subseteq}  \supp(\vu)$. Similarly, we have $\supp(\vy) {\subseteq}  \supp(\vv)$.

Furthermore, suppose that   $\A$ is reducible.  Then either we have case {(i)}, i.e., there {is a} nonempty proper index subset $J_x{\subsetneq}  [m]$ and a  proper index {$j\in[n]$} 
such that
\begin{equation}\label{equ:x_reducible}
	{a_{i_2ji_1j}+	a_{i_1ji_2j}}=0, \ \forall i_1\in J_x, \forall i_2\notin J_x.
\end{equation}
or we have case {(ii)}, i.e., there is a  proper index {$i\in[m]$} 
and a nonempty  proper index subset  $J_y{\subsetneq}  [n]$ such that
\begin{equation}\label{equ:y_reducible}
	{a_{ij_1ij_2}+a_{ij_2ij_1}}=0, \ \forall j_1\in J_y, \forall j_2\notin J_y.
\end{equation}

Suppose that we have case {(i)}.
Let $\vx$ be defined by $x_i=0$ if $i\in J_x$ and $x_i=1$ if $i\notin J_x$, and {$\vy=\ve_j(n)$.}
Then  we may verify that {$u_i=0$ for any $i\in J_x$, which suggests that} $\supp(\vx) =\supp(\vu)$.  Similarly, if we have case {(ii)}, then we may verify that $\supp(\vy) =\supp(\vv)$.  This proves the ``only if'' part of the second conclusion of this theorem.

We now prove the {``if''} part of the second conclusion of this theorem. Suppose that there  exists   $\vx\in\Re_+^m$ with $0< |\supp(\vx)|<m$  and $\vy\in\Re_+^n$ with $0< |\supp(\vy)|\le n$  such that either $\supp(\vx) =\supp(\vu)$ or $\supp(\vy) =\supp(\vv)$. In the former case, let $J_x=[m]\setminus \supp(\vx)$ and    {$j\in \supp(\vy)$.}  It is straightforward to verify that \eqref{equ:x_reducible} holds, corresponding to case (i). Similarly, in the latter scenario, we can confirm that case (ii) applies. In either case, we conclude that   $\A$ is reducible.    This proves the {``if''}  part of the second conclusion of this theorem,
thereby concludes  the proof.
\end{proof}

Viewing the proof of Theorem \ref{Thm_nonegative_supp}, we say that $\A$ is $x$-partially  reducible if there
{is a}     nonempty proper index subset  $J_x{\subsetneq}  [m]$ and {an index $j\in[n]$}  such that  \eqref{equ:x_reducible} holds.  Otherwise, we say that $\A$ is $x$-partially  irreducible.  Similarly,  we say that $\A$ is $y$-partially  reducible if there
{is a} nonempty proper index subset    $J_y{\subsetneq} [n]$ {an index $i\in[m]$} such that   \eqref{equ:y_reducible} holds.  Otherwise, we say that $\A$ is $y$-partially  irreducible.  Then $\A$ is reducible if and only if it is either $x$-partially  reducible or $y$-partially  reducible.

\begin{Cor}\label{cor:irreduc_iff}
Suppose that $\A = \left(a_{i_1j_1i_2j_2}\right) \in NBQ(m, n)$, where $m,n\ge 2$. Let $\vu$ and $\vv$ be defined as in Theorem \ref{Thm_nonegative_supp}. Then we have
\begin{itemize}
	\item[(i)] $\A$ is $x$-partially  irreducible  if and only if for any $\vx\in\Re_+^m$ with $0< |\supp(\vx)|<m$  and $\vy\in\Re_+^n$ with $0< |\supp(\vy)|\le n$, we have
	\begin{equation*}
		\supp(\vx) {\subsetneq} \supp(\vu).
	\end{equation*}
	
	\item[(ii)] $\A$ is $y$-partially  irreducible  if and only if  for any $\vx\in\Re_+^m$ with $0<|\supp(\vx)|\le m$  and $\vy\in\Re_+^n$ with $0<|\supp(\vy)| <n$, we have
	\begin{equation*}
		\supp(\vy)  {\subsetneq} \supp(\vv).
	\end{equation*}
	
	\item[(iii)] $\A$ is  irreducible  if and only if items (i) and (ii) hold simultaneously.
\end{itemize}
\end{Cor}
\begin{proof}
This  is the contrapositive statement of  Theorem~\ref{Thm_nonegative_supp}.
\end{proof}

\begin{Cor}\label{cor:x^n-1>0}
Suppose that $\A = \left(a_{i_1j_1i_2j_2}\right) \in NBQ(m, n)$,  where $m,n\ge 2$.  Let $\vx^{(0)}\in\Re_+^m {\setminus \0_m}$ and   $\vy^{(0)}\in\Re_+^n {\setminus \0_n}$,
\begin{equation*}
	\vx^{(k)} ={1 \over 2}{(\A+\I)}\vx^{(k-1)}\vy^{(k-1)}\cdot \vy^{(k-1)} + {1 \over 2}{(\A+\I)}\cdot \vy^{(k-1)}\vx^{(k-1)}\vy^{(k-1)},
\end{equation*}
and
\begin{equation*}
	\vy^{(k)} =  {1 \over 2}{(\A+\I)} \vx^{(k-1)}\vy^{(k-1)}\vx^{(k-1)}\cdot + {1 \over 2} {(\A+\I)} \vx^{(k-1)}\cdot\vx^{(k-1)}\vy^{(k-1)},
\end{equation*}
for any $k=1,\dots,\max\{m-1,n-1\}$. If $\A$ is $x$-partially irreducible, then $\vx^{(m-1)}>\0_m$. Similarly,  if $\A$ is $y$-partially irreducible,  then  $\vy^{(n-1)}>\0_n$.
\end{Cor}
\begin{proof}
By Corollary~\ref{cor:irreduc_iff}, we have $\left| \supp\left(\vx^{(k)}\right)\right|>\left| \supp\left(\vx^{(k-1)}\right)\right|$ as long as $\supp\left(\vx^{(k-1)}\right)\neq [m]$ and $\left| \supp\left(\vy^{(k)}\right)\right|>\left| \supp\left(\vy^{(k-1)}\right)\right|$ as long as $\supp\left(\vy^{(k-1)}\right)\neq [n]$. Thus, we have $|\supp\left(\vx^{(m-1)}\right)|=m$ and $|\supp\left(\vy^{(n-1)}\right)|=n$. This completes the proof.
\end{proof}

Based on Corollary~\ref{cor:x^n-1>0}, a necessary and sufficient condition for the verification of a nonnegative irreducible biquadratic tensor is presented in the following theorem.
\begin{Thm}\label{thm:verif_irredu}
Suppose that $\A = \left(a_{i_1j_1i_2j_2}\right) \in NBQ(m, n)$, where $m,n\ge 2$.
Let $\vx^{(i,0)} ={\vx^{(i,j,0)}}= \ve_i(m)$, $\vy^{(j,0)} ={\vy^{(i,j,0)}}= \ve_j(n)$,
\begin{equation*}
	{\vx^{(i,j,k)}} ={1 \over 2}(\A+\I) \vx^{(i,j,k-1)} \vy^{(j,0)}\cdot\vy^{(j,0)} + {1 \over 2}(\A+\I)  \cdot\vy^{(j,0)}\vx^{(i,j,k-1)}\vy^{(j,0)},
\end{equation*}
and
\begin{equation*}
	{\vy^{(i,j,k)}} ={1 \over 2} (\A+\I)  \vx^{(i,0)} \vy^{(i,j,k-1)}\vx^{(i,0)}\cdot + {1 \over 2}(\A+\I)  \vx^{(i,0)} \cdot\vx^{(i,0)}\vy^{(i,j,k-1)},
\end{equation*}
for any {$k=\max\{m,n\}$,} $i\in[m]$ and $j\in[n]$. Then $\A$ is $x$-partially irreducible if and only if $\vx^{(i,j,m-1)}>\0_m$ for all  $i\in[m]$ and $j\in[n]$.
Similarly, 	 $\A$ is $y$-partially irreducible if and only if $\vy^{(i,j,n-1)}>\0_n$ for all  $i\in[m]$ and $j\in[n]$.
\end{Thm}
\begin{proof}
The necessity part follows from Corollary~\ref{cor:x^n-1>0}. We will now demonstrate the sufficiency aspect utilizing the method of contradiction.
Suppose that $\A$ is $x$-partially reducible.   Then there are    two nonempty proper index subsets $J_x{\subsetneq}  [m]$ and $\red{j\in} [n]$ such that  \eqref{equ:x_reducible} holds. Let $i\in [m]\setminus J_x$. We may see that $x_s^{(i,j,k)}=0$ for all $s\in J_x$ and all $k$. This leads to a
contradiction with the assumption. Thus, the sufficiency is proved.
\end{proof}

Theorem~\ref{thm:verif_irredu} establishes a verification framework for assessing the irreducibility  of nonnegative biquadratic tensors by irreducibility  of  matrices associated with $A_{\vx}^{(j,j)}=\A(:,j,:,j)\in \Re^{m\times m}$ and   $A_{\vy}^{(i,i)}=\A(i,:,i,:)\in \Re^{n\times n}$ for \red{all} $i\in[m]$ and $j\in[n]$.

{\subsection{M$^+$-eigenvalues}
We now study existence and properties of M$^+$-eigenvalues of nonnegative biquadratic tensors.
}

Let $\A = \left( a_{i_1j_1i_2j_2}\right) \in NBQ(m, n)$.
{We define the associated} biquadratic polynomial as
$$f(\x, \y) := \langle \A, \x \circ \y \circ \x \circ \y \rangle \equiv \sum_{i_1, i_2 =1}^m \sum_{j_1, j_2 = 1}^n a_{i_1j_1i_2j_2}x_{i_1}y_{j_1}x_{i_2}y_{j_2}.$$
Let $\lambda_{\max}(\A)$ be the largest M-eigenvalue of $\A$.  Then we have
\begin{equation} \label{lamax}
\lambda_{\max}(\A) = \max \{ f(\x, \y) : \x^\top \x = \y^\top \y = 1, \x \in \Re^m, \y \in \Re^n \}.
\end{equation}
The M-spectral radius of $\A$,  denoted by  $\rho_M(\A)$, is defined as the largest absolute value among all M-eigenvalues of $\A$.  With this definition, we present the following theorem.

\begin{Thm}\label{Thm:rho=lmd_max}
Let $\A = \left( a_{i_1j_1i_2j_2}\right) \in NBQ(m, n)$, where $m, n \ge 2$.  Then we have
\begin{equation} \label{rholamax}
	\rho_M(\A) = \lambda_{\max}(\A),
\end{equation}
and $\lambda_{\max}(\A)$ is an M$^+$-eigenvalue of $\A$. {Consequently,   $\A$ has at least one  M$^+$-eigenvalue.}
\end{Thm}	
\begin{proof}  Let $\lambda$ be an arbitrary M-eigenvalue of $\A$, with M-eigenvectors $\x$ and $\y$.  Then by (\ref{lamax}), we have
$$|\lambda| = |f(\x, \y)| \le f(|\x|, |\y|) \le \lambda_{\max}(\A).$$
This shows (\ref{rholamax}).

Now, assume that the M-eigenvalue $\lambda_{\max}(\A)$ has a pair of M-eigenvectors $\x^*$ and $\y^*$.  Then we have
$$\lambda_{\max}(\A) = |f(\x^*, \y^*)| \le f(|\x^*|, |\y^*|) \le \lambda_{\max}(\A).$$
This implies that $f(|\x^*|, {|\vy^*|}) = \lambda_{\max}(\A)$, i.e., the nonnegative vector {pair} $|\x^*|$ and $|\y^*|$   {is a pair of M$^+$-eigenvectors} of $\A$.  This proves the second conclusion and completes the proof.
\end{proof}	

{
A nonnegative biquadratic tensor always has an M$^+$ eigenvalue, yet it may   have no M$^{++}$ eigenvalue, as illustrated in the following example.
\begin{example}
	Let $\A\in NBQ(2,2)$ by defined by $a_{1111}=1$ and all other entries are zeros. Then $\A$ has two eigenvalues 1 and 0, with the corresponding eigenvectors
	\begin{eqnarray*}
		& \left\{\begin{array}{c}
			\vx = (1.0000,\   0),^\top\\
			\vy = (1.0000,\    0),^\top\\
		\end{array}\right.&
		\left\{\begin{array}{c}
			\vx = (0,\    1.0000 ),^\top\\
			\vy = (0,\    1.0000).^\top\\
		\end{array}\right.
	\end{eqnarray*}
	Consequently, $\A$ has no M$^{++}$ eigenvalue.
\end{example}

However, if the nonnegative biquadratic tensor is irreducible, then it always has an M$^{++}$ eigenvalue, as illustrated in the following theorem.
}

\begin{Thm}\label{Thm:eigpair_positive}
Suppose that $\A = \left(a_{i_1j_1i_2j_2}\right) \in NBQ(m, n)$ is irreducible, where $m,n\ge 2$.  
{Assume that} $\lambda$  is  an M$^+$ eigenvalue of $\A$  with nonnegative M-eigenvector  $\{\bar \vx, \bar \vy\}$.    Then $\lambda$ is a positive M$^{++}$-eigenvalue.
{Consequently,   $\A$ has at least one  M$^{++}$-eigenvalue.}
\end{Thm}
\begin{proof}
By the definition of M-eigenvalues and M-eigenvectors, we have
$${\A\bar\vx\bar \vy \cdot\bar \vy +}\A\cdot\bar \vy \bar \vx\bar \vy={2}\lambda \bar\vx \text{ and }\A\bar \vx\bar\vy \bar \vx\cdot +\A\bar \vx\cdot \bar \vx\bar \vy={2}\lambda \bar\vy.$$ We now show that $\bar \vx> \0_m$ and $\bar \vy > \0_n$ utilizing the method of contradiction.

Suppose that $\bar\vx \ngtr \0_m$. Let $J_x=[m]\setminus \supp(\bar\vx)$. Then $J_x$ is a nonempty set. For any $i\in J_x$, we have
\begin{eqnarray*}
	& & {\sum_{i_1=1}^m \sum_{j_1,j_2=1}^n a_{i_1j_1ij_2} \bar x_{i_1}\bar y_{j_1}\bar y_{j_2} +} \sum_{i_2=1}^m \sum_{j_1,j_2=1}^n a_{ij_1i_2j_2} \bar x_{i_2}\bar y_{j_1}\bar y_{j_2} \\
	&=& {\sum_{i_1\notin J_x} \sum_{j_1,j_2\in\supp(\bar\vy)} a_{i_1j_1ij_2} \bar x_{i_1}\bar y_{j_1}\bar y_{j_2} +}\sum_{i_2\notin J_x} \sum_{j_1,j_2\in\supp(\bar\vy)} a_{ij_1i_2j_2} \bar x_{i_2}\bar y_{j_1}\bar y_{j_2}\\
	&=& 0.
\end{eqnarray*}
This shows $a_{ij_1i_2j_2}=0$ for all  $i\in J_x$,  $i_2\notin J_x$, and $j_1,j_2\in\supp(\bar\vy)$. Thus, $\A$ is $x$-partially reducible.  This leads to a contradiction. Hence $\bar \vx> \0_m$. Similarly, we could show that $\bar \vy > \0_n$. Therefore, we have $\lambda > 0$ and is an M$^{++}$-eigenvalue.
\end{proof}

The {irreducibility} condition is important in Theorem~\ref{Thm:eigpair_positive}, as demonstrated in the following example.
{
\begin{example} Consider the  diagonal nonnegative biquadratic tensor $\A = \left(a_{i_1j_1i_2j_2}\right)$ $\in NBQ(m, n)$, where each entry  $a_{i_1j_1i_2j_2}$ is nonzero  only if $i_1=i_2$ and $j_1=j_2$.
	Since $\A$ is diagonal, it is reducible. For any  $i\in[m]$ and $j\in[n]$, $\lambda = a_{ijij}\in \Re_+$  forms an   M$^{+}$ eigenvalue of $\A$ with the corresponding eigenvector being $\vx =\ve_i(m)$ and $\vy = \ve_j(n)$. However, it is not an M$^{++}$ eigenvalue. 	
\end{example}
}

These two theorems form the weak Perron-Frobenius theorem of nonnegative biquadratic tensors.

\section{A Max-Min Characterization and A Min-Max Characterization}

For any $\vx\in \Re^m$ and $\vy\in \Re^n$, let
\begin{equation}\label{equ:gh}
\vg =\frac12 ( \A \cdot \vy\vx\vy + \A \vx \vy\cdot\vy), \ \ \vh=\frac12 ( \A \vx\cdot\vx\vy + \A \vx  \vy\vx\cdot).	
\end{equation}
We now introduce the concept of non-degeneracy, which encompasses irreducibility as a special case.
\begin{Def}[non-degenerate]
Let $\A\in NBQ(m,n)$. We say $\A$ is non-degenerate if
{no pair $(x_i,g_i)$ or $(y_j,h_j)$ is zero simultaneously for all $i\in[m]$ and $j\in[n]$.}
\end{Def}

By the proof in Theorem~\ref{Thm:eigpair_positive},
{for all $i\in[m]$ and $j\in[n]$, if $x_i=0$ and $y_j=0$,  then there is $g_i\neq 0$ and $h_j\neq 0$. Therefore,}
irreducible biquadratic tensors are  non-degenerate.

Denote
$${S_+^{m-1}}=\{\vx\in\Re_+^m: \sum x_i^2=1\}$$
as the nonnegative section of the unit sphere surface  in the $m$-dimensional space and
$$   S_{++}^{m-1}=\{\vx\in\Re_{++}^m: \sum x_i^2=1\}$$
as the interior set of ${S_+^{m-1}}$.
Let $\A\in NBQ(m,n)$ be non-degenerate. We define the following two functions for all  $\vx \in \Re_+^m\setminus \{\0_m\}$ and $\vy \in \Re_+^n\setminus \{\0_n\}$.
\begin{eqnarray}\label{equ:uv}
v(\vx,\vy) = \min_{i\in[m],j\in [n]} \left\{\frac{g_i}{x_i}, \frac{h_j}{y_j}\right\}, \ \
u(\vx,\vy) = \max_{i\in[m],j\in [n]} \left\{\frac{g_i}{x_i}, \frac{h_j}{y_j}\right\}.
\end{eqnarray}
Then $v(\vx,\vy) \le u(\vx,\vy)$, and both $u(\vx,\vy)$ and $v(\vx,\vy)$ are well-defined  over the extended reals for all  $\vx \in \Re_+^m\setminus \{\0_m\}$ and $\vy \in \Re_+^n\setminus \{\0_n\}$.
Furthermore, both $u(\vx,\vy)$ and $v(\vx,\vy)$ are continuous in $\vx\in   S_{++}^{m-1}$ and $\vy\in   S_{++}^{n-1}$. We also define
\begin{equation}\label{equ:rho_*^*}
\rho_* = \inf_{\vx\in   S_{++}^{m-1}, \vy\in   S_{++}^{n-1}} u(\vx,\vy) \text{ and } \rho^* = \sup_{\vx\in   S_{++}^{m-1}, \vy\in   S_{++}^{n-1}} v(\vx,\vy)
\end{equation}
The following lemma shows that both   $\rho_*$ and $\rho^*$  are attainable.

\begin{Lem}\label{Lem:rho_acheied}
Let $\A\in NBQ(m,n)$ be non-degenerate. Then there exist  two vectors $\vz_*=(\vx_*,\vy_*), \vz^*=(\vx^*,\vy^*)\in S_+^{m-1}\otimes S_+^{n-1}$ such that
$\rho_*=u(\vx_*,\vy_*)$ and $\rho^*=v(\vx^*,\vy^*)$.
\end{Lem}
\begin{proof}
By definition, the function $v(\vx,\vy)$ is  nonnegative and continuous in $S_+^{m-1} \otimes S_+^{n-1}$. Therefore,  there exists $\vz^*=(\vx^*,\vy^*)\in S_+^{m-1}\otimes S_+^{n-1}$ such that
$\rho^*=\sup_{\vx\in   S_{++}^{m-1}, \vy\in   S_{++}^{n-1}} v(\vx,\vy)=\max_{\vx\in S_+^{m-1}, \vy\in S_+^{n-1}} v(\vx,\vy)=v(\vx^*,\vy^*) < \infty$.

Similarly, the function $u(\vx,\vy)$ is  nonnegative and lower semi-continuous in $S_+^{m-1} \otimes S_+^{n-1}$, which is bounded below by $\rho_*$. By the compactness of $S_+^{m-1} \otimes S_+^{n-1}$, there exists $\vz_*=(\vx_*,\vy_*)\in S_+^{m-1}\otimes S_+^{n-1}$ such that
$\rho_*=\min_{\vx\in S_+^{m-1}, \vy\in S_+^{n-1}} u(\vx,\vy)$ $=u(\vx_*,\vy_*)$.
\end{proof}

By Example 2 of \cite{DLQY20}, for an irreducible nonnegative biquadratic tensor $\A$, there may exist additional M$^+$-eigenvalues, which are not equal to $\lambda_{\max}(\A)$.
\begin{example}\cite{DLQY20}\label{ex:DLQY}
Let $\B\in NBQ(2,2)$ be defined by $b_{1111} = 4$,  $b_{1212} = b_{2121} = 10$, $b_{2222}=2$, $b_{1112}=b_{1121}=b_{1211}=b_{2111}=1$, $b_{1122}=b_{1221} =b_{2112}=b_{2211}=1$, and $b_{1222}=b_{2212}=b_{2122}=b_{2221}=2$.

It is a nonnegative irreducible symmetric biquadratic tensor. By the closed form in \cite{QDH09}, we obtain its M-eigenvalues  as
\[ 10.9075,\ 10.9075,\ 10.5000,\ 5.5925,\    5.5925,\         4.8202,\    3.7408,\    1.2332,\]
and the corresponding eigenvectors as
\begin{eqnarray*}
	& \left\{\begin{array}{c}
		\vx = (0.2936,\    0.9559)^\top,\\
		\vy = (0.9442,\  0.3294)^\top,\\
	\end{array}\right.&
	\left\{\begin{array}{c}
		\vx = (0.9442,\  0.3294)^\top,\\
		\vy = (0.2936,\    0.9559)^\top,\\
	\end{array}\right. \\
	& \left\{\begin{array}{c}
		\vx = (0.7071,\    0.7071)^\top,\\
		\vy = (0.7071,\    0.7071)^\top,\\
	\end{array}\right.&
	\left\{\begin{array}{c}
		\vx = (-0.7934,\    0.6087)^\top,\\
		\vy = (0.7699,\    0.6381)^\top,\\
	\end{array}\right. \\
	& \left\{\begin{array}{c}
		\vx = (0.7699,\    0.6381)^\top,\\
		\vy = (-0.7934,\   0.6087)^\top,\\
	\end{array}\right.&
	\left\{\begin{array}{c}
		\vx = (-0.8111,\    0.5849)^\top,\\
		\vy = (-0.8111,\    0.5849)^\top,\\
	\end{array}\right. \\
	& \left\{\begin{array}{c}
		\vx = (-0.9910,\    0.1336)^\top,\\
		\vy = (-0.9910,\    0.1336)^\top,\\
	\end{array}\right.&
	\left\{\begin{array}{c}
		\vx = (-0.1908,\    0.9816)^\top,\\
		\vy = (-0.1908,\    0.9816)^\top.\\
	\end{array}\right. \\
\end{eqnarray*}
\end{example}

The biquadratic tensor $\B$ in Example~\ref{ex:DLQY} has {two distinct} M$^{++}$ eigenvalues  $10.9075$    and $10.5000$.
Therefore, $\rho_* \le 10.5000 < 10.9075 \le \rho^*$,
which shows   the  minimax   {characterization of irreducible nonnegative matrices} is not true {for irreducible nonnegative biquadratic tensors} in the traditional sense.

However, in the following, we show the minimax theorem is  still true in a  general sense: the max-min characterization holds for $\lambda_{\max}(\A)$, and the min-max characterization holds for $\lambda^{+}_{\min}(\A)$, where $\A \in NBQ(m, n)$, $\lambda_{\max}(\A)$ is the largest M$ ^{+}$-eigenvalue, i.e., the largest M-eigenvalue of $\A$, and $\lambda^{+}_{\min}(\A)$ is defined as the smallest M$^+$ eigenvalue of $\A$.

\begin{Thm}\label{thm:lmd_bd}
If $\A\in NBQ(m,n)$  is irreducible, then for any M$^+$-eigenvalue $\lambda$ of $\A$, we have
\begin{equation}\label{equ:lmd_bd}
	0 < \rho_*\le \lambda \le \rho^*.
\end{equation}
\end{Thm}
\begin{proof}
{It follows \eqref{equ:rho_*^*} that $\rho_*\ge0$.} We now show $\rho_*>0$ by contradiction.
If $ \rho_*=0$, then there exist vectors $\vx\in S_+^{m-1}$ and $\vy\in S_+^{n-1}$ such that $\vg=\0_m$ and  $\vh=\0_n$.
Suppose $\vg=\0_m$.
If there exists $i$ such that $x_i=0$, we may derive
\begin{equation*}
	a_{ij_1i_2j_2} + a_{i_2j_1ij_2} =0,\ \forall i_2\in \supp(\vx), j_1,j_2\in \supp(\vy).
\end{equation*}
Consequently, $\A$ is $x$-reducible; if $\vx>\0_m$, we may derive {$\A(:,j,:,j)$} is a zero matrix for all $j\in\supp(\vy)$, which derives that $\A$ is  $x$-reducible. In both case, it contradicts with $\A\in NBQ(m,n)$  is irreducible. Therefore, $\vg\neq \0_m$. Similarly, we may show $\vh\neq \0_n$. This shows $\rho_*>0$.

For any M$^+$-eigenvalue $\lambda$ of $\A$ with eigenvector pair $\vx \ge \0_m,\vy \ge \0_n$, we may show that {$\lambda=u(\vx,\vy)\ge\rho_*$ and $\lambda=v(\vx,\vy)\le\rho^*$.} This completes the proof.
\end{proof}
The assumption that $\A\in NBQ(m,n)$  is irreducible is necessary in Theorem~\ref {thm:lmd_bd}. Otherwise, we have the following example.

\begin{example}
Let $\A\in NBQ(2,2)$ with
\[a_{1122}=a_{1221}=a_{2112}=a_{2211}=\frac12\]
and all other elements are zeros.
Then $\A$ is both $x$-reducible and $y$-reducible. Consequently, we may verify that $\lambda=0$ is an M$^+$ eigenvalue of $\A$ with eigenvector $\vx=(1,0)^\top, \vy = (1,0)^\top$.

By direct computation, we have $\vg=(x_2y_1y_2, x_1y_1y_2)^\top$ and $\vh=(x_1x_2y_2, x_1x_2y_1)^\top$.
Furthermore, we may derive that
\begin{eqnarray*}
	\rho_* &=& \min_{{\vx\in S_{++}^1, \vy\in S_{++}^1}} \max \left\{\frac{g_1}{x_1}, \frac{g_2}{x_2}, \frac{h_1}{y_1}, \frac{h_2}{y_2}\right\} \\
	&=& \min_{{\vx\in S_{++}^1, \vy\in S_{++}^1}} \frac{g_1}{x_1}\  \mathrm{ s.t. } \  \frac{g_1}{x_1} \le \min \left\{\frac{g_2}{x_2}, \frac{h_1}{y_1}, \frac{h_2}{y_2} \right\} \\
	&=& \min_{{\vx\in S_{++}^1, \vy\in S_{++}^1}}  \frac{x_2y_1y_2}{x_1}  \ \mathrm{ s.t. }  \ x_1^2 \le \min \left\{x_2^2, y_1^2, y_2^2 \right\} \\
	&=& \frac12,
\end{eqnarray*}
where the second equality follows from the symmetry of this problem.
Similarly, we may derive that $\rho^*=\frac12$.
Thus, \eqref{equ:lmd_bd} does not hold.
\end{example}

{It follows from Example~\ref{ex:DLQY} that the  M$^+$ eigenvalue of irreducible nonnegative biquadratic tensors may not be unique.}
We now derive a sufficient condition that guarantees the uniqueness of  M$^+$ eigenvalues. We begin with the following lemma.

\begin{Lem}\label{Lem:Kron_Meig}
Suppose $\A\in NBQ(m,n)$, and there exist two symmetric  matrices  $B\in\Re^{m\times m}$ and $C\in \Re^{n\times n}$ such that  $\A=B\otimes C$, i.e., the entries of $\A$ satisfy \begin{equation}\label{equ:A=B_kron_C}
	\frac14a_{ijkl}+\frac14a_{kjil}+\frac14a_{ilkj}+\frac14a_{klij}=b_{ik}c_{jl}
\end{equation}
for all $i,k\in[m]$ and $j,l\in[n]$. Then $(\lambda,\vx,\vy)$ is a nonzero M eigenpair of $\A$ if and only if $(\alpha,\vx)$ and $(\beta,\vy)$ are the nonzero   eigenpairs of $B$ and $C$ respectively, and $\lambda=\alpha\beta$.
\end{Lem}
\begin{proof}
By our assumption, we have
\[\left\{ \begin{array}{c}
	\frac12(\A \cdot \vy\vx\vy + \A \vx \vy\cdot\vy) =\left(\vy^\top C\vy\right)B\vx,\\\frac12 ( \A \vx\cdot\vx\vy + \A \vx  \vy\vx\cdot) = \left(\vx^\top B\vx\right)C\vy.
\end{array}\right.\]

On one hand, suppose that equations~\eqref{e5}-\eqref{e7} hold.
From these, we deduce the following system:
\[\left(\vy^\top C\vy\right)B\vx =\lambda \vx \text{ and } \left(\vx^\top B\vx\right)C\vy=\lambda \vy.\]
Left-multiplying the first equation by $\vx^\top$, and the second equation by $\vy^\top$, respectively, we have $\lambda = \left(\vy^\top C\vy\right)\left(\vx^\top B\vx\right)$.
If $\lambda\neq 0$, it follows that $\vy^\top C\vy\neq 0$ and $\vx^\top B\vx\neq 0$. Consequently, the above equations simplify to  $B\vx = \left(\vx^\top B\vx\right)\vx$ and $C\vy=\left(\vy^\top C\vy\right) \vy$,
which establishes the necessary part.

On the other hand, suppose that $B\vx=\alpha \vx$ and $C\vy=\beta \vy$. Then we have
\[\left\{ \begin{array}{c}
	\frac12(\A \cdot \vy\vx\vy + \A \vx \vy\cdot\vy) =\left(\vy^\top C\vy\right)B\vx =\alpha\beta\vx,\\\frac12 ( \A \vx\cdot\vx\vy + \A \vx  \vy\vx\cdot) = \left(\vx^\top B\vx\right)C\vy=\alpha\beta\vy,
\end{array}\right.\]
which establishes the sufficiency part and completes the proof.
\end{proof}

\begin{Thm}\label{Thm:Kron_Meig}
Suppose $\A\in NBQ(m,n)$ is irreducible, and there exist two  nonnegative and irreducible symmetric  matrices  $B\in\Re^{m\times m}$ and $C\in \Re^{n\times n}$ such that equation \eqref{equ:A=B_kron_C} holds.
Then the M$^+$ eigenvalue of $\A$ is unique. 	
\end{Thm}
\begin{proof}
Under the above assumption, both $B\in\Re^{m\times m}$ and $C\in \Re^{n\times n}$ possess a unique nonnegative eigenvector, which is strictly positive. Combining this result with Lemma~\ref{Lem:Kron_Meig} derives that $\A$ has one unique M$^+$ eigenvalue.
\end{proof}

\begin{Thm}
If $\A\in NBQ(m,n)$  is irreducible, then
\[\rho_* = \lambda^{+}_{\min}(\A) \text{ and }  \rho^*=\lambda_{\max}(\A) =\rho_M(\A).\]
\end{Thm}
\begin{proof}
It follows from Theorem~\ref{Thm:rho=lmd_max} that 	$\rho_M(\A) = \lambda_{\max}(\A)$.
By combining the conclusions in Theorem~\ref{thm:lmd_bd}, it remains to show $\rho_* \ge \lambda_{\min}(\A)$ and $\rho^*\le \lambda_{\max}(\A)$.

By Lemma~\ref{Lem:rho_acheied}, there exists $\vz^*=(\vx^*,\vy^*)\in S_+^{m-1}\otimes S_+^{n-1}$ such that
$\rho^*=\max_{\vx\in S_+^{m-1}, \vy\in S_+^{n-1}}v(\vx^*,\vy^*)$. Therefore, $g_i\ge \rho^* x_i^*$ and  $h_j\ge \rho^* y_j^*$ {for all $i\in[m]$ and $j\in[n]$, and} $\A\vx\vy\vx\vy = \vg^\top \vx=\vh^\top \vy\ge \rho^*$.
Thus, we have $\rho^*\le \lambda_{\max}(\A)$.

Similarly, we could show $\rho_* \ge \lambda^+_{\min}(\A)$ and completes the proof.
\end{proof}

Hence, if $\lambda^+_{\min}(\A) = \lambda_{\max}(\A)$, i.e., all the M$^+$-eigenvalues are equal, then the min-max theorem still holds in the traditional sense.   Otherwise, the max-min characterization holds for $\lambda_{\max}(\A)$, while the min-max characterization holds for $\lambda^+_{\min}(\A)$.   In this sense, not only the largest M$^+$-eigenvalue, but also the smallest M$^+$-eigenvalue, play important roles for a nonnegative biquadratic tensor.

We  have the following lemma.

\begin{Lem}
Suppose that  $\A\in NBQ(m,n)$  is irreducible and $\lambda_0$ is  an M$^+$-eigenvalue of $\A$    associated with  the  eigenvector pair $(\vx_0,\vy_0)$.  If
\begin{equation}\label{cond:unique}
	\rho^* \le \lambda_0\le \rho_*,
\end{equation}
then $\lambda_0$ is the unique M$^+$-eigenvalue of $\A$.
\end{Lem}
\begin{proof}
By combining \eqref{cond:unique} with Theorem~\ref{thm:lmd_bd}, we obtain  $\rho^* \le \lambda_0\le \rho_* \le \rho^*$. This implies $\lambda_0=\rho_*=\rho^*$, which shows that $\lambda_0$ is the unique M$^+$-eigenvalue of $\A$. This completes the proof.
\end{proof}

In Example \ref{ex:DLQY}, the nonnegative irreducible biquadratic tensor has eight M-eigenvalues.  Three of them are M$^+$-eigenvalues.  The other five M-eigenvalues are still positive, but smaller than the three M$^+$-eigenvalues.  However, M-eigenvalues of nonnegative irreducible biquadratic tensors can be negative.   In \cite{QDH09}, there is such an example.  Furthermore, for a nonnegative tensor $\A$, it is possible that there is a positive M-eigenvalue $\mu_1$, which is not an M$^+$-eigenvalue, yet
$\lambda^+_{\min}(\A) < \mu_1 < \lambda_{\max}(\A)$.  It is also possible that there is a negative M-eigenvalue $\mu_2$, satisfying $\lambda^+_{\min}(\A) < |\mu_2| < \lambda_{\max}(\A)$.  In fact, we may have such $\mu_1$ and $\mu_2$ in one example.   See
the following example.

\begin{example}\label{ex:multi_M+eig}
Let $\A\in NBQ(2,2)$ with
\[a_{1111} = 1, a_{1211}=a_{1112}=0, a_{1212}=1, a_{1122}=a_{1221}=a_{2112}=a_{2211}=2, \]
\[ a_{1121} =a_{2111}=2,a_{1222}=a_{2212}=0,a_{2121}=2, a_{2122}=a_{2221}=0, a_{2222}=1.\]
It is a nonnegative reducible symmetric biquadratic tensor. By the closed form in \cite{QDH09}, we obtain its M-eigenvalues  as
\[4.6312,\    2.3970,\    1.7917,\    1.0000,\   -0.1142,\   -1.9038\]
and the corresponding eigenvectors as
\begin{eqnarray*}
	& \left\{\begin{array}{c}
		\vx = (0.6639,\    0.7478),^\top\\
		\vy = (0.8774,\    0.4798),^\top\\
	\end{array}\right.&
	\left\{\begin{array}{c}
		\vx = (-0.6577,\    0.7533),^\top\\
		\vy = (-0.5762,\    0.8173),^\top\\
	\end{array}\right. \\
	& \left\{\begin{array}{c}
		\vx = (-0.1048,\    0.9945),^\top\\
		\vy = (-0.8848,\    0.4660),^\top\\
	\end{array}\right.&
	\left\{\begin{array}{c}
		\vx = (0,\    1.0000 ),^\top\\
		\vy = (0,\    1.0000),^\top\\
	\end{array}\right. \\
	& \left\{\begin{array}{c}
		\vx = (0.7405,\    0.6720),^\top\\
		\vy = (-0.4884,\    0.8726),^\top\\
	\end{array}\right.&
	\left\{\begin{array}{c}
		\vx = (-0.7433,\    0.6690),^\top\\
		\vy = (0.8250,\    0.5651).^\top\\
	\end{array}\right. \\
\end{eqnarray*}

Then we have the largest M-eigenvalue $\lambda_{\max}(\A) = 4.6312$ and the smallest M$^+$-eigenvalue $\lambda^+_{\min}(\A) = 1.0000$.   They are the M$^+$ eigenvalues of $\A$. However, $\A$ still has {four} other M-eigenvalues, which are not M$^+$-eigenvalues.   Among them, $2.3970$ and $1.7917$ are both greater than $\lambda^+_{\min}(\A) = 1.0000$, and a negative M-eigenvalue {$-1.9038$,} whose absolute value is still
greater than $\lambda^+_{\min}(\A) = 1.0000$.
\end{example}

{\section{A  Collatz Method}}

In this section, we generalize the Collatz method for calculating the spectral radius of irreducible nonnegative matrices \cite{Gv96} and tensors \cite{NQZ09}  to biquadratic tensors.

The algorithm is summarized in Algorithm~\ref{Alg:Collatz_NBQ}.

\begin{algorithm}[t]
\caption{A Collatz method for computing M$^+$-eigenvalues of irreducible nonnegative biquadratic tensors} \label{Alg:Collatz_NBQ}
\begin{algorithmic} [1]
	\Require $\A\in NBQ(m,n)$, $\vx^{(0)}\in \Re_{+}^m$, $\vy^{(0)}\in \Re_{+}^n$, $k_{\max}$,  $\epsilon$. Let $\underline{\lambda}^{(-1)}=-1$ and $\bar{\lambda}^{(-1)}=\infty$.
	\For{$k=0,\dots,k_{\max},$}
	\State Compute  $\vg^{(k)} =\frac12 ( \A \cdot \vy^{(k)}\vx^{(k)}\vy^{(k)} + \A \vx^{(k)} \vy^{(k)}\cdot\vy^{(k)})$ and $\vh^{(k)}=\frac12 ( \A \vx^{(k)}\cdot\vx^{(k)}\vy^{(k)} + \A \vx^{(k)}  \vy^{(k)}\vx^{(k)}\cdot).$	
	\State Let  $\underline{\lambda}^{(k)} = \min_{x_i^{(k)}>0, y_j^{(k)}>0} \left\{\frac{g_i^{(k)}}{x_i^{(k)}}, \frac{h_j^{(k)}}{y_j^{(k)}}\right\}$ and $\bar{\lambda}^{(k)} = \max_{x_i^{(k)}>0, y_j^{(k)}>0} \left\{\frac{g_i^{(k)}}{x_i^{(k)}}, \frac{h_j^{(k)}}{y_j^{(k)}}\right\}$.
	\If{$\min\{ |\underline{\lambda}^{(k)}-\bar{\lambda}^{(k)}|, |\underline{\lambda}^{(k)}-\underline{\lambda}^{(k-1)}|+ |\bar{\lambda}^{(k)}-\bar{\lambda}^{(k-1)}| \}\le \epsilon $,}
	\State Stop.
	\EndIf
	\State Let $\vx^{(k+1)}=\frac{\vg^{(k)}}{\|\vg^{(k)}\|}$ and $\vy^{(k+1)}=\frac{\vh^{(k)}}{\|\vh^{(k)}\|}$.
	\EndFor
	\State \textbf{Output:} $\vx^{(k)}$ and $\vy^{(k)}$
\end{algorithmic}
\end{algorithm}

\section{Numerical Results}
To demonstrate the efficiency of Algorithm~\ref{Alg:Collatz_NBQ}, we conduct a series of numerical experiments.
We generate   $\vx\in \Re^m$, $\vy\in \Re^n$ randomly by the  uniform  distribution and then normalize them to get the initial points by $\vx^{(0)} = \frac{\vx}{\|\vx\|}$ and  $\vy^{(0)} = \frac{\vy}{\|\vy\|}$, respectively.
The parameters in  Algorithm~\ref{Alg:Collatz_NBQ} are set as follows:    $k_{\max}=1000$,  $\epsilon = 10^{-8}$.
To ensure robustness, each experiment is repeated  100 times with different initial points, and the average results are reported.
In the following tables, `Iter' denotes the average iteration to reach the stopping criterion, `Time (s)' reports the average CPU time consumed in seconds, `$\underline{\lambda}$' and `$\bar{\lambda}$' denote the values of   $\underline{\lambda}^{(k)}$  and $\bar{\lambda}^{(k)}$ at the final iterations, respectively,  $\lambda =\frac12 (\underline{\lambda}+\bar{\lambda})$,
`Res' denotes the  residue computed by
$$Res =\left\|\begin{bmatrix}
\frac12 \A\cdot \vy\vx\vy+\frac12 \A\vx \vy\cdot\vy-\lambda\vx \\
\frac12 \A\vx \cdot\vx\vy+\frac12 \A\vx \vy\vx\cdot-\lambda\vx \\
\end{bmatrix}\right\|_{\infty},$$
where $\vx$ and $\vy$ denote the values of   $\vx^{(k)}$  and $\vy^{(k)}$ at the final iterations, respectively.
We also use `ratio$_{\underline{\lambda}\approx \rho_M}$` and `ratio$_{\bar{\lambda}\approx\rho_M}$' to denote the ratio of $|\underline{\lambda}-\rho_M|\le 10^{-6}$ and $|\bar{\lambda}-\rho_M|\le 10^{-6}$  over 100 repeats, respectively.

First, we examine two small-scale cases in Examples~\ref{ex:DLQY} and \ref{ex:multi_M+eig}, which serve as preliminary illustrations of the algorithm's performance. The numerical results presented in Table~\ref{tab:ex6.1} demonstrate that Algorithm~\ref{Alg:Collatz_NBQ} is capable of accurately computing the largest  M$^+$ eigenvalue, although there are multiple M$^+$ eigenvalues for those two examples.

\begin{table}
\begin{center}
	\caption{Numerical results of Algorithm~\ref{Alg:Collatz_NBQ} for two small examples.}\label{tab:ex6.1}
	\begin{tabular}{c|c|c|c|c|c|c|c}
		\hline
		& $\rho_M$ & Iter & Time (s) & $\bar{\lambda}-\underline{\lambda}$ & Res &  ratio$_{\underline{\lambda}\approx \rho_M}$ &  ratio$_{\bar{\lambda}\approx\rho_M}$   \\ \hline
		Example~\ref{ex:DLQY} & 10.9075 & 60.18 & 4.65e-02 & 2.30e-08 & 1.10e-08 & 100\% & 100\% \\
		Example~\ref{ex:multi_M+eig} &  4.6312 & 24.22 & 1.88e-02 & 7.89e-09 & 2.95e-09 & 100\% & 100\% \\
		\hline
	\end{tabular}
\end{center}
\end{table}

We further validate the algorithm through randomized experiments. Specifically, we construct symmetric nonnegative biquadratic tensors by sampling from a uniform distribution and apply Algorithm~\ref{Alg:Collatz_NBQ} to compute the largest  M$^+$ eigenvalues.
The numerical results, as shown in Table~\ref{tab:ex6.2}, indicate that Algorithm~\ref{Alg:Collatz_NBQ} achieves high precision in computing the largest  M$^+$, typically converging within approximately 10 iterations and requiring less than one second of computation time. On average, both the eigenvalue gap
$\bar{\lambda}-\underline{\lambda}$  and  the residue in eigenpair are maintained below  $10^{-9}$.
Furthermore, we denote  $\rho_M$ by the largest M$^+$ eigenvalues from 100 repeated experiments. The results demonstrate that both the lower and upper bounds consistently converge to the same M$^+$
eigenvalue, confirming the reliability and stability of the proposed method.

\begin{table}
\begin{center}
	\caption{Numerical results of Algorithm~\ref{Alg:Collatz_NBQ} for randomly generated biquadratic tensors.}\label{tab:ex6.2}
	\setlength{\tabcolsep}{10pt}
	\begin{tabular}{c|c|c|c|c|c|c|c}
		\hline
		$m$ & $n$ & Iter & Time (s) & $\bar{\lambda}-\underline{\lambda}$ & Res &  ratio$_{\underline{\lambda}\approx \rho_M}$ &  ratio$_{\bar{\lambda}\approx\rho_M}$  \\ \hline
		\multirow{3}*{10}
		& 10  & 10.41  & 8.86e$-$03 & 5.03e$-$10  & 8.10e$-$11 & 100\%  & 100\% \\	
		& 50  & 9.00  & 1.64e$-$02 & 4.33e$-$11  & 4.77e$-$12  & 100\%  & 100\% \\
		& 100  & 8.06  & 2.48e$-$02 & 2.21e$-$10  & 1.45e$-$11 & 100\%  & 100\% \\ \hline
		\multirow{3}*{30}
		& 10  & 9.05  & 1.24e$-$02 & 2.67e$-$10  & 3.55e$-$11 & 100\%  & 100\% \\
		& 50  & 8.00  & 5.60e$-$02 & 8.48e$-$12  & 7.09e$-$13 & 100\%  & 100\%  \\
		& 100  & 7.80  & 1.83e$-$01 & 2.37e$-$11  & 1.84e$-$12 & 100\%  & 100\%  \\  \hline
		\multirow{3}*{50}
		& 10  & 9.00  & 1.82e$-$02 & 6.13e$-$11  & 5.49e$-$12 & 100\%  & 100\%  \\
		& 50  & 7.89  & 1.26e$-$01 & 1.53e$-$11  & 1.09e$-$12 & 100\%  & 100\%  \\
		& 100  & 7.00  & 4.24e$-$01 & 1.77e$-$11  & 1.17e$-$12 & 100\%  & 100\%  \\
		\hline
	\end{tabular}
\end{center}
\end{table}

\section{Final Remarks}

The {M-eigenvalue problem of} nonnegative biquadratic {tensors} is somewhat similar with the Z-eigenvalue problem of nonnegative cubic tensors \cite{CPZ13}, while the  singular value problem of  nonnegative biquadratic tensors \cite{CQZ10} is  similar with the H-eigenvalue problem of nonnegative cubic tensors \cite{CPZ08, CPZ11, FGH13, QL17}.
In this paper, we introduced M$^+$-eigenvalues and M$^{++}$-eigenvalues for nonnegative biquadratic tensors, and studied their properties.  In contrast to its singular value counterpart,
the largest M$^+$-eigenvalue of a nonnegative biquadratic tensor is not unique. {We also established a sufficient condition ensuring the uniqueness of the M$^+$-eigenvalue.}
An important further research problem is 
{to explore more conditions for the uniqueness of}
the M$^+$-eigenvalue.  The other further research problems are as follows.  If the nonnegative biquadratic tensor is primitive \cite{CPZ11, QL17} or even positive, can we obtain some better results? If
the nonnegative biquadratic tensor is only weakly irreducible \cite{FGH13}, can we still obtain the results?
{Although extensive numerical experiments   suggest the convergence of the Collatz algorithm, we do not yet have a rigorous proof. Thus, we propose its convergence as an open problem for further investigation.}

\bigskip	


{{\bf Acknowledgment}}
This work was partially supported by Research  Center for Intelligent Operations Research, The Hong Kong Polytechnic University (4-ZZT8),    the National Natural Science Foundation of China (Nos. 12471282 and 12131004), the R\&D project of Pazhou Lab (Huangpu) (Grant no. 2023K0603), and the Fundamental Research Funds for the Central Universities (Grant No. YWF-22-T-204).



{{\bf Data availability} Data will be made available on reasonable request.

{\bf Conflict of interest} The authors declare no conflict of interest.}




\end{document}